\documentclass[11pt,a4paper]{amsart}
\usepackage{amsthm,amsmath,amscd}
\title{Properties of the resolution of almost Gorenstein algebras}
\author{Giuseppe Zappal\`a}

\subjclass[2010]{13 H 10, 14 N 20, 13 D 40}

\DeclareSymbolFont{rsfscript}{OMS}{rsfs}{m}{n}
\DeclareSymbolFontAlphabet{\mathrsfs}{rsfscript}

\DeclareSymbolFont{AMSb}{U}{msb}{m}{n}
\DeclareSymbolFontAlphabet{\mathbb}{AMSb}

\DeclareSymbolFont{eufrak}{U}{euf}{m}{n}
\DeclareSymbolFontAlphabet{\gothic}{eufrak}

\newcommand\CI{\operatorname{CI}}

\newcommand\rank{\operatorname{rank}}

\newcommand\rw{\Rightarrow}
\newcommand\bu{\bullet}

\newcommand\rh{\rightarrow}

\newcommand\Gor{\operatorname{Gor}}

\newcommand\zz{{\mathbb Z}}

\newcommand{\bo}{\bigoplus}
\newcommand{\op}{\oplus}
\newcommand{\vt}{\vartheta}
\newcommand{\ord}{\operatorname{ord}}

\newtheorem{thm}{Theorem}[section]

\newtheorem{prp}[thm]{Proposition}
\newtheorem{cor}[thm]{Corollary}

\theoremstyle{definition}

\newtheorem{dfn}[thm]{Definition}

\theoremstyle{remark}
\newtheorem{rem}[thm]{Remark}
\newtheorem{exm}[thm]{Example}

\newcounter{num}

\begin{document}

%\maketitle
%\footnotemark
%\footnotetext[1]{2000 {\it Mathematics subject classification: }13 D 40, 13 H 10.}
%\subjclass{13 D 40, 13 H 10}

%\begin{abstract} 
%\end{abstract}

\begin{abstract}
We study properties of the resolution of almost Gorenstein artinian algebras $R/I,$
i.e. algebras defined by ideals $I$ such that $I=J+(f),$ with $J$ Gorenstein ideal and $f\in R.$ 
Such algebras generalize the well known almost complete intersection artinian algebras. 
Then we give a new explicit description of the resolution and of the graded Betti numbers of 
almost complete intersection ideals of codimension $3$ and we characterize the ideals whose graded 
Betti numbers can be achieved using artinian monomial ideals.
\end{abstract}

\maketitle

\section{Introduction}
\markboth{\it Introduction}{\it Introduction}
Let $I\subset R=k[x_1,\ldots,x_c]$ be a homogeneous ideal. In commutative algebra and algebraic geometry it is a very important subject to understand the structure of the graded minimal free resolution of $I.$ The simplest case is when $I$ is generated by a regular sequence. In this case the Koszul complex provides the answer to this problem, in any codimension. In the other cases we know satisfying answers only in few situations. 
\par
In codimension $2,$ when $I$ is a perfect ideal, 
the structure of the resolution is given by the Hilbert-Burch theorem (see \cite{Hi}, \cite{Bu}).
This allowed, for instance, to characterize the graded Betti numbers of such ideals (see \cite{Ga}, \cite{CGO}).
\par
In codimension $3,$ when $I$ is a Gorenstein ideal, 
the structure of the resolution is given by the Eisenbud-Buchsbaum theorem (see \cite{BE}).
Also in this case, due to this very important result, it is possible to characterize the graded Betti numbers 
(see \cite{Di}, \cite{RZ2}, \cite{CV}).
\par
When $I$ is perfect of codimension greater than $2$ or $I$ is Gorenstein of codimension greater than $3$ very little is known about the structure the resolution. One of the most studied cases are Gorenstein ideals in codimension $4$ (see \cite{KM},\cite{Re}), although it is not known a description of the graded Betti numbers.
\par
Another very important case consists of perfect ideals of codimension $c$ generated by $c+1$ generators, i.e. almost complete intersection ideals. These ideals have the advantage that are directly linked to a Gorenstein ideal in a complete intersection, so there is the expectation that properties of the Gorenstein resolutions can affect their resolution. In codimension $3$ the structure of the resolution of such ideals was studied in \cite{BE} and in \cite{RZ4}.
\par
In section $3$ we study the resolution of the almost Gorenstein algebras, a very important generalization of the notion of almost complete intersection. Also these ideals are directly linked to the Gorenstein ideals. In particular, in Theorem \ref{fc}, we state an important duality properties for the graded minimal free resolution.
\par
In section $4$ we focus on codimension $3$ almost complete intersection algebras and we give an explicit new characterization of the graded Betti numbers. Finally we characterize the ideals, in this class, whose graded Betti numbers can be achieved using monomial ideals 
(Corollary \ref{mon1} and Proposition \ref{mon2}).

\section{Preliminary facts} 
\markboth{\it Preliminary facts}
{\it Preliminary facts}
Throughout the paper $k$ will be a infinite field and $R:=k[x_1,\ldots,x_c],$ $c\ge 3,$ will be the standard graded polynomial $k$-algebra.
We consider only homogeneous ideals.
\par
Let $A=R/I$ be a standard graded artinian $k$-algebra, where $I$ is a homogeneous ideal of $R.$ Recall that $A$ is said a 
{\em almost complete intersection algebra} if $I$ is minimally generated by $c+1$ elements (in this case $I$ is said an almost complete intersection ideal). Every almost complete intersection ideal is directly linked to a Gorenstein ideal in a complete intersection.
\par
It is possible to generalize this notion in the following way.
\begin{dfn}
Let $A_Q=R/I_Q$ be a standard graded artinian $k$-algebra, where $I_Q$ is a homogeneous ideal of $R.$ 
We say that $A$ is an {\em almost Gorenstein algebra} if 
\begin{itemize}
\item[1.] $I_Q=I_Z+(f),$ where $I_Z$ is a Gorenstein ideal of codimension $c$ and $f\in (x_1,\ldots,x_c),$ $f\not\in I_Z.$
\item[2.] $A_Q$ is not a Gorenstein algebra.
\end{itemize}
\end{dfn}
In particular, if $I_Z$ is a complete intersection ideal of codimension $c$, then $A_Q$ is an almost complete intersection algebra.
\par
Gorenstein ideals of codimension $3$ are better understood, because of the structure theorem of Buchsbaum and Eisenbud (see \cite{BE}).
The possible graded Betti numbers for Gorenstein ideal of codimension $3$ were described in \cite{Di} and 
relatively to the Hilbert function in \cite{RZ2}. 
\par
More precisely let $(d_1,\ldots, d_{2n+1}),$ $d_1 \le\ldots\le d_{2n+1},$ be a sequence of positive integers. It is the sequence of the degrees of the minimal generators of a Gorenstein ideal of codimension $3$ iff
\begin{itemize}
\item [1)] $\vt=\frac{1}{n} \sum_{i=1}^{2n+1} d_i$ is an integer;
\item [2)] $\vt > d_i + d_{2n+3-i}$ for $2 \le i \le n$ (Gaeta conditions).
\end{itemize}

Let $\delta=(d_1,\ldots, d_{2n+1}),$ $d_1 \le\ldots\le d_{2n+1},$ be the sequence of the degrees of the minimal generators of a Gorenstein ideal of codimension $3.$ 
%We define
%  $$B=\{3\le i\le n+1\mid\vt\le d_i+d_{2n+4-i}\}$$
%and
% $$C=\{4\le i\le n+2\mid\vt\le d_i+d_{2n+5-i}\}.$$
Let $\Gor(\delta)$ be the set of all Gorenstein ideals $I \subset k[x_1,x_2,x_3],$ whose 
sequence of the degrees of the minimal generators is $\delta.$
Moreover we define 
\begin{multline*}
  \CI_{\delta}=\{(a_1,a_2,a_3)\in(\zz^{+})^3\mid a_1 \le a_2 \le a_3 \mbox{ and }\exists I \in \Gor(\delta) \mbox{ containing} \\
	\mbox{a regular sequence of degrees } (a_1,a_2,a_3)\}.
\end{multline*}
To study ideals linked to Gorenstein ideals in a complete intersection, in codimension $3,$ we need the following result.

\begin{thm}\label{c3b}
The set $\CI_{\delta}$ has a unique minimal element.
\end{thm}
\begin{proof}
See Theorem 3.6 in \cite{RZ3}, where this minimal element is explicitly computed.
\end{proof}
We will denote the only minimal element of $\CI_{\delta}$ with $\min(\delta).$

%\begin{thm}\label{c3b}
%Let $\beta=\big((d_1,\ldots, d_{2n+1}),(\vartheta-d_{2n+1},\ldots,\vartheta-d_1),(\vartheta)\big),$
%$n\vartheta=\sum d_i,$ $d_1\le\ldots\le d_{2n+1},$ be a Betti sequence admissible for an Artinian Gorenstein quotient of $k[x,y,z]$ and $B$ %and $C$ as above. Then $\CI^g_{\beta}$ has a unique minimal element. Precisely, 
%\begin{itemize}
%\item [i)] if $B \ne \emptyset,$ the minimal element is $(d_1,d_{\max B}, d_{2n+4-\min B});$
%\item [ii)] if $B = \emptyset$ and $C \ne \emptyset,$ the minimal element is $(d_1,d_2, d_{\max C});$
%\item [iii)] if $B = \emptyset$ and $C = \emptyset,$ the minimal element is $(d_1,d_2, d_3).$
%\end{itemize}
%\end{thm}
%\begin{proof}
%See Theorem 3.6 in \cite{RZ3}.
%\end{proof}

\section{Almost Gorenstein algebras}
\markboth{\it Almost Gorenstein algebras}{\it Almost Gorenstein algebras}
Let $A_Q=R/I_Q$ be an almost Gorenstein algebra. Then $I_Q=I_Z+(f),$ where $I_Z$ is a Gorenstein ideal. Let $e$ be the socle degree of $A_Z=R/I_Z.$
Since $f\in (x_1,\ldots,x_c),$ $f\not\in I_Z$ we get that $1\le \deg f\le e.$ We are interested on the graded minimal free resolutions of such ideals. We denote by $\vartheta_Z$ the degree of the only last syzygy of $I_Z.$ Of course $\vartheta_Z=e+c.$

\begin{prp}
Let $A_Q=R/I_Q$ be an almost Gorenstein algebra, with $I_Q=I_Z+(f)$ as above. Then $I_G=I_Z:I_Q$ is a Gorenstein ideal of codimension $c$ and $\vartheta_G=\vartheta_Z-\deg f.$
\end{prp} 
\begin{proof}
%ote that $I_G=I_Z:I_Q=I_Z:(f).$ 
Let
 $$0 \rh F_c \rh F_{c-1} \rh \ldots \rh F_1 \rh R \rh A_Q \rh 0$$
and
 $$0 \rh K_c \rh K_{c-1} \rh \ldots \rh K_1 \rh R \rh A_Z \rh 0$$
be graded minimal free resolutions. Since $A_Z$ is Gorenstein $K_c \cong R(-\vartheta_Z).$ 
Moreover $F_1 \cong K_1'\oplus R(-d^*),$ where $K_1'$ is a direct summand of $K_1$ and $d^*=\deg f.$
Since $I_Z\subseteq I_Q$ we can use liason theory. If we set $A_G=R/I_G$, we get the exact sequence
  $$0 \rh (A_Q)^{\vee}(-\vartheta_Z) \rh A_Z \rh A_G \to 0.$$
So we can compute a graded free resolution of $A_G$ by mapping cone. Finally we can delete the summand $K_1'^{\vee}(-\vartheta_Z)$ in the tail of resolution, so the last module in a minimal graded free resolution of $A_G$ is isomorphic to $R(-\vartheta_Z+d^*).$
\end{proof}

\begin{rem} \label{remd}
Note that $\vartheta_Z-d^*=\vartheta_G.$
We set $d=2\vartheta_Z-\vartheta_G.$ Note that $d=\vartheta_Z+d^*=\vartheta_G+2d^*,$ where $d^*=\deg f.$
\end{rem}

\begin{prp} \label{resq}
A graded resolution (possibly not minimal) of $A_Q$ is
 \begin{multline*}
  0 \rh G_{c-1}(-d^*) \rh G_{c-2}(-d^*)\oplus K_{c-1} \rh \ldots \\
	\ldots \rh G_{1}(-d^*)\oplus K_2 \rh R(-d^*)\oplus K_{1} \rh R \rh A_Q \rh 0.
 \end{multline*}
\end{prp}
\begin{proof}
In general we have the short exact sequence
  $$0 \rh (A_G)^{\vee}(-\vartheta_Z) \rh A_Z \rh A_Q \to 0.$$
Let $G_{\bu}$ and $K_{\bu}$ be graded minimal free resolution of $A_G$ and $A_Z$ respectively. By the duality properties of the Gorenstein resolutions we have that 
  $$G_{\bu}\cong G_{\bu}^{\vee}(-\vartheta_G).$$
Consequently, by Remark \ref{remd}
  $$G_{\bu}(-d^*)\cong G_{\bu}^{\vee}(-\vartheta_G-d^*)\cong G_{\bu}^{\vee}(-\vartheta_Z).$$
So a resolution (which may be not minimal) of $A_Q$ is
 \begin{multline*}
  0 \rh G_{c-1}(-d^*) \rh G_{c-2}(-d^*)\oplus K_{c-1} \rh \ldots \\
	\ldots \rh G_{1}(-d^*)\oplus K_2 \rh R(-d^*)\oplus K_{1} \rh R \rh A_Q \rh 0.
 \end{multline*}
\end{proof}

One of the simplest cases is when $\deg f=e.$ 

\begin{cor}
Let $A_Q=R/I_Q$ be an almost Gorenstein algebra, with $I_Q=I_Z+(f)$ and $\deg f=e,$ the socle degree of $A_Z.$ Let 
 $$0 \rh R(-(e+c)) \rh K_{c-1} \rh \ldots \rh K_1 \rh R \rh A_Z \rh 0$$
be a graded minimal free resolution of $A_Z.$
Let us suppose that $I_Z$ is generated in degrees less than $e.$
\par
Then a graded minimal free resolution of $A_Q$ is
 $$0 \rh R(-(e+c-1))^{\binom{c}{c-1}} \rh K_{c-1}\oplus R(-(e+c-2))^{\binom{c}{c-2}} \rh \ldots $$
 $$\ldots \rh K_1\oplus R(-(e-1))^{\binom{c}{1}} \rh K_1\oplus R(-e) \rh R \rh A_Q \rh 0.$$
\end{cor} 
\begin{proof}
If we set $I_G=I_Z:I_Q,$ then $I_G=(x_1,\ldots,x_c).$ Hence 
  $$G_i(-d^*)=G_i(-e) \cong R(-(i+e))^{\binom{c}{i}}.$$ 
Now we apply Proposition \ref{resq} to obtain a graded resolution of $A_Q.$
Since $I_Z$ is generated in degrees less than $e,$ we have that $A_{Q,i}=A_{Z,i},$ for $i \le e-1$ and $A_{Q,i}=0$ for $i \ge e,$
so such a resolution is minimal.
\end{proof}

If we set $H_i=G_{i}(-d^*),$ for $1 \le i \le c-1$ we have that 
  $$H_i=G_{i}(-d^*)\cong G_{c-i}^{\vee}(-\vartheta_G-d^*)\cong $$ 
	$$\cong (H_{c-i}(d^*))^{\vee}(-\vartheta_G-d^*)\cong H_{c-i}^{\vee}(-\vartheta_G-2d^*)\cong H_{c-i}^{\vee}(-d).$$
So the resolution of Proposition \ref{resq} becomes
 \begin{multline*}
  0 \rh H_{c-1} \rh H_{c-2}\oplus K_{c-1} \rh \cdots \\
	\cdots \rh H_{c-2}^{\vee}(-d)\oplus K_3 \rh H_{c-1}^{\vee}(-d)\oplus K_2 \rh R(-d^*)\oplus K_{1} \rh R \rh A_Q \rh 0.
 \end{multline*}
%Now we set $d=2\vartheta_Z-\vartheta_G.$ Note that $d=\vartheta_Z+d^*=\vartheta_G+2d^*,$ where $d^*=\deg f.$

\begin{thm} \label{fc}
Let $A_Q=R/I_Q$ be an almost Gorenstein algebra with graded minimal free resolution
   $$0 \rh F_c \rh F_{c-1} \rh \ldots \rh F_2 \rh F_1 \rh R \rh A_Q \rh 0.$$
Let us suppose that a minimal set of generators of $I_Q$ is $(f_1,\ldots,f_n,f)$ with $(f_1,\ldots,f_n)$ minimal set of generators of a  Gorenstein ideal $I_Z.$ 
Let $K_{\bu}$ a graded minimal free resolution of $I_Z$ and $G_{\bu}$ a graded minimal free resolution of $I_G=I_Z:I_Q.$ 
Then 
   $$F_2 \cong F_c^{\vee}(-d)\oplus F_2' \cong G_1(-d^*)\oplus K_2'$$
for some $F_2'$ and where $K_2'$ is a direct summand of $K_2.$
\end{thm} 
\begin{proof}
%Let $I_G=I_Z:I_Q.$ Let $G_{\bu}$ a graded minimal free resolution of $I_G.$ 
By Proposition \ref{resq},
we have that $F_c$ is a direct summand of $H_{c-1}=G_{c-1}(-d^*),$ namely $H_{c-1}\cong F_c\oplus H_{c-1}'.$
By the hypotheses about the sets of generators $F_1\cong R(-d^*)\oplus K_{1}$ and 
$F_2\cong G_1(-d^*)\oplus K_{2}'=H_1\oplus K_{2}'$ where $K_2'$ is a direct summand of $K_2.$ Therefore
  $$F_2 \cong H_1\oplus K_{2}'\cong H_{c-1}^{\vee}(-d)\oplus K_{2}'\cong F_c^{\vee}(-d)\oplus H_{c-1}'^{\vee}(-d)\oplus K_{2}'.$$
\end{proof}

\begin{cor} 
Under the same hypotheses of Theorem \ref{fc}, 
  $$F_2 \cong F_c^{\vee}(-d)\oplus G_1'(-d^*)\oplus K_{2}',$$
where $G_1'$ is a direct summand of $G_1$ and $K_{2}'$ is a direct summand of $K_2.$
\end{cor} 
\begin{proof}
This follows immediately by the conclusion of the proof of Theorem \ref{fc}.
\end{proof}

\begin{cor}
Let $A_Q=R/I_Q$ be an almost complete intersection algebra with graded minimal free resolution
   $$0 \rh F_c \rh F_{c-1} \rh \ldots \rh F_2 \rh F_1 \rh R \rh A_Q \rh 0.$$
Then $F_2 \cong F_c^{\vee}(-d)\oplus F_2',$ where $d$ is the sum of the degrees of a minimal set of generators of $I_Q.$
 \end{cor} 
\begin{proof}
By the hypotheses $I_Q$ is minimally generated by $(f_1,\ldots,f_c,f),$ where $(f_1,\ldots,f_c)$ is a regular sequence. 
We set $I_Z=(f_1,\ldots,f_c).$ Now, by applying Theorem \ref{fc}, we get $F_2 \cong F_c^{\vee}(-d)\oplus F_2'.$ 
By Remark \ref{remd}, $d=\vartheta_Z+\deg f,$ hence $d=\sum_{i=1}^c\deg f_i+\deg f.$
\end{proof}

This Corollary generalizes Lemma 2.1 in \cite{RZ1}.

\section{Almost complete intersections in codimension $3$}
\markboth{\it Almost complete intersections in codimension $3$}{\it Almost complete intersections in codimension $3$}

In this section we study some properties of the resolutions of the almost complete intersection algebras in codimension $3.$ 
\par
An almost complete intersection algebra $A_Q=R/I_Q$ of codimension $3$ has a graded minimal free resolution of the type
  $$0 \rh F_3 \rh F_2 \rh F_1 \rh R \rh A_Q \rh 0.$$
We set $t_Q=\rank F_3.$ Of course $t_Q\ge 2.$

\begin{thm} \label{str3}
Let $A_Q=R/I_Q$ be an almost complete intersection algebra in codimension $3.$ 
\par
If $t=t_Q$ is even, then a graded minimal free resolution of $A_Q$ is
 \begin{multline*}
  0 \rh \bigoplus_{i=1}^{t}R(-(d-s_i)) \rh \bigoplus_{i=1}^{3}R(-(d_i+d^*))\oplus\bigoplus_{i=1}^{t}R(-s_i) \rh \\
	\rh R(-d^*)\oplus\bigoplus_{i=1}^{3}R(-d_i) \rh R \rh A_Q \rh 0
 \end{multline*}
where $d=d_1+d_2+d_3+d^*,$ $d_1\le d_2\le d_3,$ $s_1\le\ldots\le s_t.$
\par
If $t=t_Q$ is odd, then a graded minimal free resolution of $A_Q$ is
 \begin{multline*}
  0 \rh \bigoplus_{i=1}^{t}R(-(d-s_i)) \rh \bigoplus_{1\le i<j\le 3}R(-(d_i+d_j))\oplus\bigoplus_{i=1}^{t}R(-s_i) \rh \\
	\rh R(-d^*)\oplus\bigoplus_{i=1}^{3}R(-d_i) \rh R \rh A_Q \rh 0
 \end{multline*}
where $d=d_1+d_2+d_3+d^*,$ $d_1\le d_2\le d_3,$ $s_1\le\ldots\le s_t.$
\end{thm}
\begin{proof}
Just using Theorem \ref{fc} and Lemma 2.1 in \cite{RZ1}.
\end{proof}

With the notation of Theorem \ref{str3}, we set $s_Q=\sum_{i=1}^t s_i.$

\begin{prp} \label{sq}
We have that
\begin{itemize}
	\item[1)] $s_Q=(\frac{t}{2}-1)d^*+\frac{t}{2}(d_1+d_2+d_3)=d\frac{t}{2}-d^*,$ if $t=t_Q$ is even;
	\item[2)] $s_Q=\frac{t+1}{2}d^*+\frac{t-1}{2}(d_1+d_2+d_3)=d\frac{t-1}{2}+d^*,$ if $t=t_Q$ is odd.
\end{itemize}
\end{prp}
\begin{proof}
This follows from Theorem \ref{str3}, since the alternating sum of the shifts of a graded minimal free resolution 
is equal to zero.
\end{proof}

Let $A_Q$ be an almost complete intersection algebra with minimal graded free resolution $F_{\bu}.$ 
If $F_2=\bigoplus_{i=1}^{t+3}R(-\beta_i),$ we set $s_{(2)}=\sum_{i=1}^{t+3}\beta_i.$ 
Analogously if $F_3=\bigoplus_{i=1}^{t}R(-\gamma_i),$ we set $s_{(3)}=\sum_{i=1}^{t+3}\gamma_i.$
Note that $s_Q=td-s_{(3)}.$ We set
 $$u_Q=s_{(2)}-s_Q=s_{(2)}+s_{(3)}-td.$$
We can compute the fundamental numerical invariant $d^*$ directly from the graded minimal free resolution of $A_Q.$

\begin{prp}
Let $A_Q=R/I_Q$ be an almost complete intersection algebra with minimal graded free resolution $F_{\bu}.$
Let $u=u_Q$ and $d$ be the sum of the degrees of the minimal generators of $I_Q.$ Then
  $$d^*:=\frac{u-d}{2},\text{ if $t$ is even;} \,\,\,\,d^*=\frac{2d-u}{2},\text{ if $t$ is odd.}$$
\end{prp}
\begin{proof}
If $t$ is even then 
  $$F_2 \cong \bigoplus_{i=1}^{3}R(-(d_i+d^*))\oplus\bigoplus_{i=1}^{t}R(-s_i),$$
hence $u=u_Q=\sum_{i=1}^3(d_i+d^*)=d+2d^*.$
\par
If $t$ is odd then 
  $$F_2 \cong \bigoplus_{1\le i<j\le 3}R(-(d_i+d_j))\oplus\bigoplus_{i=1}^{t}R(-s_i),$$
hence $u=u_Q=\sum_{1\le i<j\le 3}(d_i+d_j)=2(d_1+d_2+d_3)=2(d-d^*).$
\end{proof}

\begin{rem}
By Theorem \ref{fc}, $F_2 \cong F_3^{\vee}(-d) \oplus F_2'.$ By Theorem \ref{str3} 
 $$F_2'\cong \bigoplus_{i=1}^{3}R(-(d_i+d^*)) \,\, \text{if $t_Q$ is even};$$
 $$F_2'\cong \bigoplus_{1\le i<j\le 3}R(-(d_i+d_j)) \,\, \text{if $t_Q$ is odd}.$$
\end{rem}

\begin{exm}
Let $A_Q$ be an almost complete intersection algebra with minimal graded free resolution
 \begin{multline*}
 0 \rh R(-20)^3 \oplus R(-21)^4 \oplus R(-22)^4 \rh \\ \rh R(-17) \oplus R(-18) \oplus R(-19)^5 \oplus R(-20)^4 \oplus R(-21)^3 \rh \\
\rh R(-8)\oplus R(-9)\oplus R(-10)\oplus R(-14) \rh R \rh A_Q \rh 0.
 \end{multline*}
We have that $t=11,$ $d=8+9+10+14=41$ and $d-20=21,d-21=20,d-22=19,$ so $u=17+18+19=54.$
Consequently
 $$d^*=\frac{2d-u}{2}=14 \rw d_1=8,d_2=9,d_3=10.$$
Moreover
 $$s_1=\ldots=s_4=19,\,s_5=\ldots=s_8=20,\,s_9=s_{10}=s_{11}=21.$$
\end{exm}

Now we compute $d^*$ for the monomial ideals.

\begin{exm} \label{mont2}
Let $I_Q=(x^{a_1},y^{a_2},z^{a_3},x^{b_1}y^{b_2}) \subset k[x,y,z],$ with $0<b_1<a_1,$ $0<b_2<a_2,$ $a_3>0.$
\par
Its graded minimal free resolution is 
 \begin{multline*}
   0 \rh R(-(a_1+b_2+a_3))\oplus R(-(b_1+a_2+a_3)) \rh R(-(b_1+a_2))\oplus R(-(a_1+b_2)) \oplus \\
	 \oplus R(-(a_1+a_3)) \oplus R(-(a_2+a_3)) \oplus R(-(b_1+b_2+a_3)) \rh \\
	 \rh R(-a_1)\oplus R(-a_2) \oplus R(-a_3) \oplus R(-(b_1+b_2)) \rh I_Q \rh 0 .
 \end{multline*}
So we have that
  $$F_2' \cong R(-(a_1+a_3)) \oplus R(-(a_2+a_3)) \oplus R(-(b_1+b_2+a_3)).$$
Consequently $d^*=a_3.$
\end{exm}

\begin{exm} \label{mont3}
Let $I_Q=(x^{a_1},y^{a_2},z^{a_3},x^{b_1}y^{b_2}z^{b_3}) \subset k[x,y,z],$ with $0<b_i<a_i,$ for $i=1,2,3.$
\par
Its graded minimal free resolution is 
 \begin{multline*}
   0 \rh R(-(a_1+a_2+b_3))\oplus R(-(a_1+b_2+a_3)) \oplus R(-(b_1+a_2+a_3)) \rh \\
	\rh R(-(b_1+b_2+a_3))\oplus R(-(b_1+a_2+b_3)) \oplus R(-(a_1+b_2+b_3)) \oplus\\
	 \oplus R(-(a_1+a_2)) \oplus R(-(a_1+a_3)) \oplus R(-(a_2+a_3)) \rh \\
	 \rh R(-a_1)\oplus R(-a_2) \oplus R(-a_3) \oplus R(-(b_1+b_2+b_3)) \rh I_Q \rh 0 .
 \end{multline*}
So we have that
  $$F_2' \cong R(-(a_1+a_2)) \oplus R(-(a_1+a_3)) \oplus R(-(a_2+a_3)).$$
Consequently $d^*=b_1+b_2+b_3.$
\end{exm}

Note that there is always a regular sequence $(f,f_2,f_3)$ in $I_Q,$ with 
 $$\deg f=d^*,\,\,\deg f_2=d_2,\,\,\deg f_3=d_3.$$
Now we study the liaison of $I_Q$ in a complete intersection of type $d^*,d_2,d_3.$

\begin{thm} \label{tpari1}
Let $A_Q=R/I_Q$ be an almost complete intersection algebra of codimension $3,$ with $t=t_Q$ even.
\par 
If $d^*\ne d_1,$ then $I_Q$ is linked to a Gorenstein ideal with $t+1$ minimal generators in a complete
intersection of type $(d^*,d_2,d_3).$
\end{thm}
\begin{proof}
We set $\vartheta_Z=d^*+d_2+d_3.$ By liaison, using mapping cone, we get the following graded resolution of a Gorenstein algebra
\begin{multline*}
 0 \rh R(-(\vartheta_Z-d_1)) \rh F_3(d_1) \oplus F_2'^{\vee}(-\vartheta_Z) \rh \\
\rh F_3^{\vee}(-\vartheta_Z) \oplus R(-d^*)\oplus R(-d_2)\oplus R(-d_3) \rh R.
\end{multline*}
We have that 
  $$F_2'^{\vee}(-\vartheta_Z) \cong \bigoplus_{i=1}^3R(-(\vartheta_Z-d_i-d^*))=R(-(d_2+d_3-d_1)) \oplus R(-d_3)\oplus R(-d_2).$$
Now we set $\vartheta_G=\vartheta_Z-d_1.$ We have that 
  $$(F_3^{\vee}(-\vartheta_Z))^{\vee}(-\vartheta_G)=F_3(d_1)$$
and 
  $$R(-d^*)^{\vee}(-\vartheta_G) \cong R(d^*-\vartheta_G) = R(-(d_2+d_3-d_1)).$$
So, if this resolution were minimal, then 
  $$R(-d_2)^{\vee}(-\vt_G)\cong R(-d_3) \iff d_2+d_3=\vt_G \iff d^*=d_1.$$
Therefore the terms $R(-d_2)$ and $R(-d_3)$ are redundant and we obtain the graded minimal free resolution
  \begin{multline*}
0 \rh R(-\vt_G) \rh F_3(d_1) \oplus R(-(d_2+d_3-d_1))\rh \\
\rh F_3^{\vee}(-\vt_Z) \oplus R(-d^*) \rh R.
  \end{multline*}
\end{proof}

\begin{exm}
Let us consider the ideal 
 $$({x}_{4}^{3}{x}_{8},{x}_{3}^{4},{x}_{4}^{3}{x}_{6}^{3},{x}_{2}^{4}{x}_{6}^{3}-{x}_{1}^{6}{x}_{8})\subset k[x_1,\ldots,x_8].$$
Its artinian reduction defines an almost complete intersection algebra whose graded minimal free resolution is
\begin{multline*}
 $$0 \rh R(-11)\op R(-14) \rh R(-7)\op R(-8)\op R(-10)^2\op R(-11) \rh \\ \rh R(-4)^2\op R(-6)\op R(-7) \rh R.$$
\end{multline*}
We see that $d^*=4.$
If we link in the complete intersection 
 $$({x}_{4}^{3}{x}_{8},{x}_{4}^{3}{x}_{6}^{3}-{x}_{3}^{4}{x}_{7}^{2},{x}_{3}^{4}{x}_{5}^{3}-{x}_{2}^{4}{x}_{6}^{3}+{x}_{1}^{6}{x}_{8})$$
we get a Gorenstein algebra whose graded minimal free resolution is
 \begin{multline*}
 $$0 \rh R(-13) \rh R(-6)\op R(-7)^2\op R(-9)\op R(-10) \rh \\ \rh R(-3)\op R(-4)\op R(-6)^2\op R(-7) \rh R,$$
\end{multline*}
so, in this case, the term $R(-6)\op R(-7)$ cannot be deleted.
\end{exm}

\begin{thm} \label{tpari2}
Let $A_Q=R/I_Q$ be an almost complete intersection algebra of codimension $3,$ with $t=t_Q$ even.
\par 
If $d^*=d_1,$ then $I_Q$ has the same graded Betti numbers of an almost complete intersection ideal $I'_Q,$ which is linked to a Gorenstein ideal with $t+1$ minimal generators in a complete intersection of type $(d^*,d_2,d_3).$
\end{thm}
\begin{proof}
By performing the liaison we get the following Gorenstein resolution as in the proof of Theorem \ref{tpari1} 
    \begin{multline*}
0 \rh R(-\vt_G) \rh F_3(d_1) \op R(-(d_2+d_3-d_1)) \op R(-d_2) \op R(-d_3) \rh \\
\rh F_3^{\vee}(-\vt_Z) \oplus R(-d^*) \op R(-d_2) \op R(-d_3) \rh R.
  \end{multline*}
Let $G_{\bu}$ be the graded minimal resolution obtained by the previous one, by omitting the term $R(-d_2) \op R(-d_3).$
By Remark 3.8 in \cite{RZ2}, there exist Gorenstein ideals whose graded minimal resolution is $G_{\bu}.$ By Theorem 3.6 in \cite{RZ3}, there exist Gorenstein ideals $I_G$ whose graded minimal resolution is $G_{\bu},$ containing a complete intersection ideal $I_Z$ of type $(d^*,d_2,d_3).$ The ideal $I'_Q=I_Z:I_G$ is an almost Gorenstein ideal with the same graded Betti numbers as $I_Q.$
\end{proof}

We have that 
  $$F_3^{\vee}(-\vt_Z) \cong \bo_{i=1}^t R(d-s_i-\vt_Z) = \bo_{i=1}^t R(-(s_i-d_1)),$$
so by Theorem \ref{tpari1} and Theorem \ref{tpari2}, $(d^*,s_1-d_1,\ldots,s_t-d_1)$ is 
a sequence of the degrees of the minimal generators of a Gorenstein ideal of codimension $3.$
\par
If $(a_1,\ldots,a_n) \in \zz^n,$ we define 
 $$\ord(a_1,\ldots,a_n)=(a_{\sigma(1)},\ldots,a_{\sigma(n)})$$ 
the permutation such that $a_{\sigma(1)} \le\ldots\le a_{\sigma(n)}.$
\par
We set $(s'_1,\ldots,s'_{t+1})=\ord(d^*+d_1,s_1,\ldots,s_t).$ So
 $$(s'_1-d_1,\ldots,s'_{t+1}-d_1)$$
is the ordered sequence of the degrees of the minimal generators of a Gorenstein ideal of codimension $3.$
\par
Moreover we need to compare $d^*$ with $d_2$ and $d_3.$ So we define
 $$p(d^*)=\begin{cases} 1 & \text{ if } d^* \le d_2 \\ 2 & \text{ if } d_2 < d^* \le d_3 \\ 3 & \text{ if } d_3 < d^*\end{cases}$$

\begin{thm} \label{strpari}
If $t$ is even, let us consider a resolution of the type
 \begin{multline*}
  0 \rh \bo_{i=1}^{t}R(-(d-s_i)) \rh \bo_{i=2}^{3}R(-(d_i+d^*))\op\bo_{i=1}^{t+1}R(-s'_i) \rh \\
	\rh R(-d^*)\oplus\bigoplus_{i=1}^{3}R(-d_i) \rh R \rh A_Q \rh 0
 \end{multline*}
where $d=d_1+d_2+d_3+d^*,$ $d_1\le d_2\le d_3,$ $s_1\le\ldots\le s_{t},$
$(s'_1,\ldots,s'_{t+1})=\ord(d^*+d_1,s_1,\ldots,s_t).$ Furthermore let
 $$(m_1,m_2,m_3)=\min (s'_1-d_1,\ldots,s'_{t+1}-d_1),\,\,m_1 \le m_2 \le m_3,$$
$p=p(d^*)$ and $(e_1,e_2,e_3)=\ord(d^*,d_2,d_3).$
\par
It is the graded minimal free resolution of an almost complete intersection algebra of codimension $3$ if and only if 
  \begin{itemize}
	  \item[1)] $\sum_{i=1}^t s_i=\frac{t}{2}d-d^*;$
    \item[2)] $d>s'_i+s'_{t+3-i},\,\,\text{for}\,\,2\le i\le t+1;$
		\item[3)] $m_p \le d^*$ and $m_i<e_i$ for $1\le i \le 3,$ $i \ne p.$
	\end{itemize}
\end{thm}
\begin{proof}
If it is the graded minimal free resolution of an almost complete intersection algebra $A_Q=R/I_Q$ of codimension $3,$ 
then the first condition holds by Proposition \ref{sq}. The second condition follows by the Gaeta conditions, i.e. we have that
\begin{multline*}
  \vt_G > s'_i-d_1+s'_{t+3-i}-d_1 \iff d^*+d_2+d_3-d_1 > s'_i-d_1+s'_{t+3-i}-d_1 \\ 
	\iff d>s'_i+s'_{t+3-i} \text{ for } 2\le i\le t+1.
\end{multline*}
The third condition follows from the fact that there exists a Gorenstein ideal $I_G$ with minimal generators 
of degrees $(s'_1-d_1,\ldots,s'_t-d_1),$ linked to $I_Q,$ in a complete intersection of type $(d^*,d_2,d_3),$
with the generator of degree $d^*$ minimal for $I_G$ and the generators of degrees $d_2$ and $d_3$ not minimal
for $I_G.$
\par
Conversely if the three conditions hold, then we can build a Gorenstein ideal $I_G$ whose minimal generators have degrees 
$(s'_1-d_1,\ldots,s'_t-d_1)$ and such that it is possible to link $I_G$ in a complete intersection $I_Z$ of type $(d^*,d_2,d_3),$
with the generator of degree $d^*$ minimal for $I_G$ and the generators of degrees $d_2$ and $d_3$ not minimal
for $I_G.$ The ideal $I_Q=I_Z:I_Q$ is the almost complete intersection ideal whose graded minimal free resolution is the assigned one.
\end{proof}

If $t$ is odd we can proceed analogously.

\begin{thm} \label{strdispari}
If $t$ is odd, let us consider a resolution of the type
 \begin{multline*}
  0 \rh \bo_{i=1}^{t}R(-(d-s_i)) \rh R(-(d_2+d_3))\op\bo_{i=1}^{t+2}R(-s'_i) \rh \\
	\rh R(-d^*)\oplus\bigoplus_{i=1}^{3}R(-d_i) \rh R \rh A_Q \rh 0
 \end{multline*}
where $d=d_1+d_2+d_3+d^*,$ $d_1\le d_2\le d_3,$ $s_1\le\ldots\le s_{t},$
$(s'_1,\ldots,s'_{t+2})=\ord(d_1+d_2,d_1+d_3,s_1,\ldots,s_t).$ Furthermore let
 $$(m_1,m_2,m_3)=\min (s'_1-d_1,\ldots,s'_{t+2}-d_1),\,\,m_1 \le m_2 \le m_3,$$
$p=p(d^*)$ and $(e_1,e_2,e_3)=\ord(d^*,d_2,d_3).$
\par
It is the graded minimal free resolution of an almost complete intersection algebra of codimension $3$ if and only if 
  \begin{itemize}
	  \item[1)] $\sum_{i=1}^t s_i=\frac{t-1}{2}d+d^*;$
    \item[2)] $d>s'_i+s'_{t+4-i},\,\,\text{for}\,\,2\le i\le t+2;$
		\item[3)] $m_p < d^*$ and $m_i \le e_i$ for $1\le i \le 3,$ $i \ne p.$
	\end{itemize}
\end{thm}
\begin{proof}
Analogous to the proof of Theorem \ref{strpari}.
\end{proof}

\begin{cor} \label{mon1}
Lat $A_Q=R/I_Q$ be an almost complete intersection algebra of codimension $3,$ whose Cohen-Macaulay type is $t_Q=2.$
\par
Let
\begin{multline*}
   0 \rh R(-(d-s_1))\oplus R(-(d-s_2)) \rh R(-s_1)\oplus R(-s_2) \oplus \\
	 \oplus R(-(d_1+d^*)) \oplus R(-(d_2+d^*)) \oplus R(-(d_3+d^*)) \rh \\
	 \rh R(-d_1)\oplus R(-d_2) \oplus R(-d_3) \oplus R(-d^*) \rh I_Q \rh 0 .
 \end{multline*}
be its graded minimal free resolution, where $d=d_1+d_2+d_3+d^*,$ $d_1\le d_2\le d_3,$ $s_1\le s_2.$
\par
Then $I_Q$ has the same graded Betti numbers of the monomial ideal
 $$J=(x^{d_2},y^{d_3},z^{d^*},x^{s_2-d_3}y^{s_1-d_2}).$$
\end{cor}
\begin{proof}
By condition $1$ in Theorem \ref{strpari}
  $$s_1+s_2=d_1+d_2+d_3 \iff (s_1-d_2)+(s_2-d_3)=d_1.$$
By condition $3$ in Theorem \ref{strpari}
  $$s_1-d_1<d_2 \text{ and } s_2-d_1<d_3$$
i.e. $$d_2+d_3-s_2<d_2 \text{ and } d_2+d_3-s_1<d_3 \iff d_3<s_2 \text{ and } d_2<s_1.$$
So we can consider the ideal $J=(x^{d_2},y^{d_3},z^{d^*},x^{s_2-d_3}y^{s_1-d_2}).$ 
Using Example \ref{mont2} we can verify that the graded Betti numbers of $J$ are the assigned ones.
\end{proof}

When the Cohen-Macaulay type is $t_Q=3,$ the graded Betti numbers of an almost complete intersection algebra not always 
can be realized by a monomial ideal.

\begin{exm}
Let us consider the ideal $I_Q=(x^2,xz+y(x+y+z),ax+by,f),$ with $a$ and $b$ general linear forms and $f$ general quintic form. It has
the following graded minimal free resolution
  \begin{multline*}
	0 \rh R(-5)^2 \op R(-7) \rh R(-4)^4 \op R(-6)^2 \rh R(-2)^3 \op R(-5) \rh I_Q \rh 0.
	\end{multline*}
We have that $d^*=5.$ In Example \ref{mont3} we computed the resolution of whatever artinian monomial almost complete intersection ideal of codimension $3$ with $t_Q=3.$
Using the same notation we must have 
  $$b_1+b_2+b_3=5 \rw b_1=1,b_2=2,b_3=2,$$
(up to change the names of $b_i$'s). So the ideal should be $$J=(x^2,y^2,z^2,xy^2z^2)=(x^2,y^2,z^2),$$ a contradiction.
\end{exm}

Now we characterize, in the case $t_Q=3,$ the graded Betti numbers which can be realized using a monomial ideal.

\begin{prp} \label{mon2}
Let $A_Q=R/I_Q$ be any almost complete intersection algebra of codimension $3,$ $t=t_Q=3.$ 
Let 
	\begin{multline*}
  0 \rh \bo_{i=1}^{3}R(-(d-s_i)) \rh \bo_{1 \le i<j \le 3}R(-(d_i+d_j))\op\bo_{i=1}^{3}R(-s_i) \rh \\
	\rh R(-d^*)\oplus\bigoplus_{i=1}^{3}R(-d_i) \rh R \rh A_Q \rh 0
 \end{multline*} 
be its graded minimal free resolution, with $d_1 \le d_2 \le d_3,$ $s_1 \le s_2 \le s_3.$
\par
Then the graded Betti numbers of $I_Q$ can be realized by an artinian monomial ideal iff
  $$d^* < s_i < d^*+d_i,\,\,\text{ for } 1 \le i \le 3.$$
\end{prp}
\begin{proof}
Let us suppose that the graded Betti numbers of $I_Q$ can be realized by a monomial ideal 
  $J=(x^{a_1},y^{a_2},z^{a_3},x^{b_1}y^{b_2}z^{b_3}) \subset k[x,y,z],$ with $0<b_i<a_i,$ for $i=1,2,3.$
We set 
 $$v_1=a_1+b_2+b_3,\,v_2=b_1+a_2+b_3,\,v_3=b_1+b_2+a_3.$$
We have that $(s_1,s_2,s_3)=\ord(v_1,v_2,v_3).$ Note that $d^*<v_i<d^*+d_i$ for $1 \le i \le 3,$
therefore $d^*<s_i$ for $1 \le i \le 3.$ Moreover $s_1\le v_1<d^*+d_1.$
\par
Now we consider $s_2.$ If $s_2=v_i$ per $i\le 2$ we deduce that 
   $$s_2=v_i<d^*+d_i\le d^*+d_2.$$
If $s_2=v_3$ then $s_3=v_i$ with $i\le 2,$ so
   $$s_2\le s_3=v_i<d^*+d_i\le d^*+d_2.$$
Finally
  $$s_3=v_i<d^*+d_i\le d^*+d_3.$$	
%\par
Conversely if $d^* < s_i < d^*+d_i,$ for $1 \le i \le 3,$ we can consider the monomial ideal
  $$J=(x^{d_1},y^{d_2},z^{d_3},x^{d^*+d_1-s_1}y^{d^*+d_2-s_2}z^{d^*+d_3-s_3}).$$
It has $4$ minimal generators since $d^*<s_i$ implies that $d^*+d_i-s_i<d_i.$
Moreover the exponents are all positive since $s_i<d^*+d_i.$ It is easy to check, using Example \ref{mont3}, 
that $J$ has the requested graded Betti numbers.
\end{proof}

\vspace{1cm}

{\sc Dipartimento di Matematica e Informatica, \\ Universit\`a di Ca\-tania, \\
                  Viale A. Doria 6, 95125 Catania, Italy}\par
{\it E-mail address: }{\tt zappalag@dmi.unict.it} 
%{\it Fax number: }{\f +39095330094}

\end{document}